\title{\bf{Algorithms for group actions in arbitrary characteristic and a problem in singularity theory}}
\author{
       \bf{ Gert-Martin Greuel and Thuy Huong Pham}\\
 }
\date{\today}

\documentclass[11pt]{article}
\usepackage{amsmath, amsthm, amssymb, mathrsfs, amscd, amsthm, amscd,amsfonts,graphicx}
\usepackage{indentfirst,url,xypic}
\usepackage{lipsum} 
\usepackage[all]{xy}
\usepackage{url}
\usepackage[colorlinks,plainpages]{hyperref}
\usepackage{verbatim}
\usepackage{fancyvrb}
\usepackage{listings}
\usepackage[colorlinks,plainpages]{hyperref}
\hypersetup{
	colorlinks=true,
	linkcolor=blue,
	filecolor=magenta,      
	urlcolor=cyan,
}

\DeclareMathOperator{\characteristic}{char}

\newtheorem{Definition}{ Definition}[section]
\newtheorem{Theorem}[Definition]{Theorem}
\newtheorem{Remark}[Definition]{Remark}

\newtheorem{Corollary}[Definition]{Corollary}
\newtheorem{Example}[Definition]{Example}

\newcommand{\R}{\mathbb{R}}

\newcommand{\C}{\mathbb{C}}

\begin{document}
\maketitle

\begin{abstract}
We consider the actions of different groups $G$ on the space ${M}_{m,n}$ of  $m\times n$ matrices with entries in the formal power series ring $K[[x_1,\ldots, x_s]], K$ an arbitrary field. $G$ acts on  $M_{m,n}$ by analytic change of coordinates, combined with the multiplication by invertible matrices from the left, the right or from both sides, respectively. This includes right and contact equivalence of functions  and mappings, resp. ideals.
$A$ is called finitely $G$--determined if any matrix $B$, with entries of $A-B$ in $\langle x_1,...,x_s\rangle ^k$ for some $k$, is contained in the $G$--orbit of $A$. The purpose of this paper is to present algorithms for checking finite determinacy, to compute determinacy bounds and to compute the image $\widetilde T_A(GA)$ of the tangent map to the orbit map $G \to GA$. The tangent image is contained in the tangent space $T_A(GA)$ of the orbit $GA$ and we apply the algorithms to prove that both spaces may be different if the field $K$ has positive characteristic, even for contact equivalence of functions. This
 fact had been overlooked by several authors before. Besides this application,  the algorithms of this paper may be of interest for the classification of singularities in arbitrary characteristic.

\end{abstract}

%%%%%%%%%%%%%%%%%%%%%%%%%%%%%%
%%%%%%%%%%%%%%%%%%%%%%%%%%%%%%
\section{Introduction}\label{introduction}
Throughout this paper let $K$ be a fixed field of arbitrary characteristic and 
\[R:=K[[{\bf{x}}]]=K[[x_1,\ldots, x_s]]\] 
the formal power series ring over $K$ in $s$ variables with maximal ideal $\mathfrak{m} = \langle x_1, \ldots , x_s \rangle$. We denote by 
\[ {M}_{m,n}:=Mat(m,n, R)\]
the set of all $m\times n$ matrices with entries in $R$. 
Let $G$ denote one of the groups $\mathcal{R},\mathcal{G}_{l},\mathcal{G}_{r},\mathcal{G}_{lr}$ (defined in section \ref {section 2}), acting on  $M_{m,n}$ by analytic change of coordinates and multiplication with invertible matrices from the left, the right or from both sides, respectively.
Two matrices $A, B \in   M_{m,n}$ are called  {\it $G$--equivalent},  denoted  $A\mathop\sim\limits^{G} B$,  if $B$ lies in the orbit of A. $A$ is said to be {\it $G$ $k$--determined} if for each matrix $B\in    M_{m,n}$ with $B-A\in \mathfrak{m}^{k+1}\cdot   M_{m,n}$, we have $B\mathop\sim\limits^{G} A$, i.e. if $A$ is $G$-equivalent to every matrix which coincides with $A$ up to and including terms of order $k$. $A$ is called {\it finitely $G$--determined} if there exists a positive integer $k$ such that it is $G$ $k$--determined.

In this paper we present algorithms for checking finite determinacy and to compute determinacy bounds. Moreover, if $A$ is finitely determined, we give algorithms to compute the image $\widetilde T_A(GA)$ of the tangent map to the orbit map $G \to GA$, which is contained in the tangent space $T_A(GA)$ of the orbit $GA$. It was discovered only recently by the authors, and announced in \cite {GP16}, that both spaces may be different if the field $K$ has positive characteristic. One of the purposes of this paper is to give a proof of this result. For this we use the above mentioned algorithms and an algorithm to compute the codimension of $T_A(GA)$ in $M_{m,n}$. Since we do not know how to compute  $T_A(GA)$ directly in positive characteristic, we present algorithms to compute the 
orbit $GA$ and the stabilizer  $G_A$.

% % % % % % % % % % % % % % % % % % % % % % % % % % % % % % % % % % % %
\section{Tangent spaces and tangent images}\label{section 2}

We review theoretical results from \cite {GP16} on tangent images and tangent spaces of the action of the group $G\in\{\mathcal{R},\mathcal{G}_{l},\mathcal{G}_{r},\mathcal{G}_{lr}\}$, with
\begin{align*}
&{\mathcal{R}}:= Aut(R),\\
&{\mathcal{G}}_l:=GL(m,R)\rtimes{\mathcal{R}},\\
&{\mathcal{G}}_r:=GL(n,R)\rtimes{\mathcal{R}}, \\
&\mathcal{G}_{lr}:=\left(GL(m,R)\times GL(n,R)\right)\rtimes \mathcal{R}.
\end{align*}
where $Aut(R)$ is the  group of $K$-algebra automorphisms of $R$.
These groups act on the space $  M_{m,n}$ as follows
\begin{align*} 
&(\phi, A)\mapsto \phi (A):=[\phi(a_{ij}({\bf{x}}))]=[a_{ij}(\phi ({\bf{x}}))],\\
%%%%%%%%%%%%%%%%
&(U,\phi, A)\mapsto U\cdot\phi (A,)\\ 
%%%%%%%%%%%%%%%%%%%% 
&(V,\phi, A)\mapsto \phi (A)\cdot V, \\
%%%%%%%%%%%%%%%%%%
&(U,V,\phi, A)\mapsto U\cdot\phi (A)\cdot V, 
%%%%%%%%%%%%%%%
\end{align*}
where ${\bf{x}}=(x_1, x_2,\ldots, x_s)$, $A=[a_{ij}({\bf{x}})]\in   M_{m,n}$, $U\in GL(m,R)$, $V\in GL(n,R)$, and $\phi ({\bf{x}}):=(\phi_1,\ldots,\phi_s)$ with $\phi_i:=\phi(x_i)\in\mathfrak{m}$ for all $i=1,\ldots,s$. 
\vskip 7pt

For $a\in  R$ and $k\in \mathbb{N}$, we by 
$jet_k(a)$ the image of $a$ in $  R/\mathfrak{m}^{k+1}$, which we identify with the power series of $a$ up to and including order $k$. 
$jet_k(A) = [jet_k(a_{ij})]$ denotes the {\textit{ k-jet}} of $A$, and
$$  M_{m,n}^{(k)}:=  M_{m,n}/\mathfrak{m}^{k+1}\cdot   M_{m,n},$$ 
the space of all $k$-jets. The $k$-jet of $G$ is
$$G^{(k)}:=\{jet_k(g)\mid g\in G\},$$
where, for example, for $g=(U,V,\phi)\in\mathcal{G}_{lr}$ we have $jet_k(g)=\left(jet_k(U), jet_k(V), jet_k(\phi)\right)$, with
$jet_k(\phi)(x_i)=jet_k(\phi(x_i))$ for all $i=1,\ldots,s$ . Then $G^{(k)}$ is an affine algebraic group, acting algebraically on the affine space $  M_{m,n}^{(k)}$ via
\begin{align*}
 &G^{(k)}\times   M_{m,n}^{(k)}\to   M_{m,n}^{(k)}, \hskip4pt
\left(jet_k(g), jet_k(A)\right)\mapsto jet_k(gA),
\end{align*}
i.e. we let representatives act and then take the $k$-jets.
\vskip 7pt

 In \cite{GP16} we  defined for $A\in \mathfrak{m}\cdot   M_{m,n}$ the following submodules of $  M_{m,n}$,
 \begin{align*}
 &\widetilde {T}_A(\mathcal{R}A):= \mathfrak{m}\cdot\left\langle\frac{\partial A}{\partial x_\nu}\right\rangle,\\
 &\widetilde T_A(\mathcal{G}_lA):= \langle E_{m, pq}\cdot A\rangle +\mathfrak{m}\cdot\left\langle\frac{\partial A}{\partial x_\nu}\right\rangle,\\
 &\widetilde T_A(\mathcal{G}_rA):=\left \langle A\cdot E_{n, hl}\right\rangle + \mathfrak{m}\cdot\left\langle\frac{\partial A}{\partial x_\nu}\right\rangle, \\
 &\widetilde T_A(\mathcal{G}_{lr}A):=\langle E_{m, pq}\cdot A\rangle +\langle A\cdot E_{n, hl}\rangle+ \mathfrak{m}\cdot\left\langle\frac{\partial A}{\partial x_\nu}\right\rangle,
 \end{align*}
the {\bf\textit{tangent images}} at $A$ to the orbit of $A$ under the actions of $\mathcal{R}$, $\mathcal{G}_{l}$, $\mathcal{G}_{r}$, and $\mathcal{G}_{lr}$ on $  M_{m,n}$, respectively. Here $\langle E_{m, pq}\cdot A\rangle$ is the $R$-submodule generated by  $E_{m, pq}\cdot A$, $p,q=1,\ldots,m$, with $E_{m,pq}$ the $(p,q)$-th canonical matrix of $Mat(m,m,R)$ (1 at place $(p,q)$ and 0 else) and $\left\langle\frac{\partial A}{\partial x_\nu}\right\rangle$ is the $R$-submodule generated by the matrices $\frac{\partial A}{\partial x_\nu} = \left[\frac{\partial a_{ij}}{\partial x_\nu} ({\bf x})\right], \nu = 1, \ldots, s$. 
 
% Replacing $\mathfrak{m}\cdot\left\langle\frac{\partial A}{\partial x_\nu}\right\rangle$ by $R\cdot\left\langle\frac{\partial A}{\partial x_\nu}\right\rangle$ in the above definition we get the {\bf\textit {extended tangent images}} $\widetilde {T}^e_A(\mathcal{R}A)$, $\widetilde T^e_A(\mathcal{G}_lA)$, $\widetilde T^e_A(\mathcal{G}_rA)$, and $\widetilde T^e_A(\mathcal{G}_{lr}A)$.\\
\medskip

It was shown in \cite{GP16} that the $k$--jet of the tangent image 
$$\widetilde T_A^{(k)}(GA):= jet_k(\widetilde T_A(GA))$$  
is the image of $T_o$, where $T_o$ is the tangent map to the orbit map $o: G^{(k)} \to G^{(k)}jet_k(A)$,
$$T_o: T_eG^{(k)}  \to  T_A^{(k)}(GA) :=  T_{jet_k(A)}\left(G^{(k)}jet_k(A)\right). $$ 
Then $\widetilde T_A(GA)$ is the inverse limit of the inverse system of $R$-modules $\widetilde T_A^{(k)}(GA)$.

Moreover, we set
	$$T_A(GA):=	
	\mathop {\lim }\limits_{\mathop {\longleftarrow} \limits_{k\ge 0} } T_A^{(k)}(GA) \subset   M_{m,n} $$ 	
and call it the {\it\bf tangent space} at $A$ to the orbit $GA$. We have $\tilde T_A(GA)\subset T_A(GA)$, and if $K$ has characteristic zero then the equality holds (see \cite{GP16}). $\tilde T_A(GA)$ is contained in $\mathfrak{m}M_{m,n}$, the tangent space of $M_{m,n}$.

\begin{Theorem}
	Let $A\in \mathfrak{m}  M_{m,n}$ and $G$ one of the groups $\mathcal{R},\mathcal{G}_{l},\mathcal{G}_{r}$, and $\mathcal{G}_{lr}$. 
	\begin{enumerate}
	\item  \rm If $\mathfrak{m}^{p+1} M_{m,n}\subset \tilde T_A(GA)$ for some $p$, then $A$ is finitely $G$-determined. Moreover, $A$ is then $G$ $(2p-ord(A)+2)$-determined, where $ord(A)$ is the minimum order of the entries of $A$.
		\item \rm  If $n=1$, i.e. $A$ is a column matrix, and $G=\mathcal{G}_{lr},\mathcal{G}_{l}$,  then $A$ is finitely $G$-determined if and only if $\mathfrak{m}^{p+1} M_{m,n}\subset \tilde T_A(GA)$ for some $p$.
	\end{enumerate}
\end{Theorem}
For the proof of 1. we refer to  \cite[Prop. 4.2]{GP16} and for 2. to \cite[Theorem 3.8]{GP17}.

% % % % % % % % % % % % % % % % % % % % % % % % % %
\section{Algorithms for the tangent image}\label{section 3}
By Theorem 2.1, the finiteness of  of the codimension of the tangent image $\widetilde T_A(GA)$, equivalent to the existence of a power $p$ of the maximal ideal $\mathfrak{m}$ such that $\mathfrak{m}^{p+1} M_{m,n}\subset \tilde T_A(GA)$, is a sufficient condition for $A$ to be finitely $G$-determined.
The following very simple algorithms compute a local standard basis of tangent image $\widetilde T_A(GA)$, a vector space basis of  $M_{m,n}/\widetilde T_A(GA)$ and the codimension of $\widetilde T_A(GA)$ in $M_{m,n}$.

Theoretically the matrix $A$ has arbitrary power series $a_{ij}$
as entries, but we assume in the algorithms below that the $a_{ij}$ are polynomials, since we can compute only with polynomial data. The output is then also polynomial. If $A$ is finitely $G$-determined then $A$ is $G$-equivalent to a matrix with polynomial entries, but finite $G$-determinacy is not assumed in the algorithms below.
By the following remark the result is nevertheless correct over the power series ring.

\begin{Remark}\rm
	Let $>$ be a local degree ordering on (the monomials of) $K[x]$  and $(>,c)$ a module ordering on $K[{\bf x}]^n$ giving priority to the monomials in  $K[x]$, see  \cite[Definition 2.3.1]{GP07}, which will also be denoted by $>$.
	
	If $N_1,\ldots, N_r$ is a standard basis w.r.t. $>$ of a submodule $M\subset K[{\bf x}]^n$, then
	\begin{enumerate}
		\item $N_1,\ldots, N_r$ generate $M\otimes_{K[{\bf x}]} K[{\bf x}]_>$ over the localization $K[{\bf x}]_>$ of $K[{\bf x}]$ w.r.t. $>$.
		\item $N_1,\ldots, N_r$ is a standard basis of $M\otimes_{K[{\bf x}]} K[[{\bf x}]]$ w.r.t. $>$ and hence generates $  M=M\otimes_{K[{\bf x}]} K[[{\bf x}]]$ over $K[[{\bf x}]]$
	\end{enumerate}
	For the proof, see \cite[Lemma 2.3.5 and Theorem 6.4.3]{GP07}  (there it is only stated for ideals but it works with the same proof for modules).
\end{Remark}
 
 For the rest of the paper we fix a local degree ordering $>$ on $K[{\bf x}]$ and a module ordering $(>,c)$ on $M_{m,n}$, also denoted by $>$ (cf. \cite {GP07}).
 
 We start with the computation of a standard basis for the tangent image w.r.t. the group 
 $\mathcal{G}_{lr}$. As the algorithms for the other groups 
$\mathcal{R}, \mathcal{G}_l, \mathcal{G}_r$ are simplifications they are omitted. 

\vskip 7pt
% % % % % % % % % %
{\bf{Algorithm 1:}} TangIm$_G$ (for $G=\mathcal{G}_{lr}$)

{\bf{Input:}} A matrix $A=[a_{ij}]\in \mathfrak{m}M_{m,n}$ (with $a_{ij}\in \langle {\bf x} \rangle K[{\bf x}]$).

{\bf{Output:}} matrices $N_1,\ldots, N_r\in M_{m,n}$, being a standard basis of the tangent image $\widetilde T_A(GA)$ w.r.t $>$.

1. Compute $M:=\langle E_{m,pq}\cdot A, p, q=1,\ldots, m\rangle + \langle A\cdot E_{n,hl}, h,l=1,\ldots,n\rangle$.

2. Compute $N:=\left\langle \frac{\partial A}{\partial x_\nu}, \nu=1,\ldots,s\right\rangle$.

3. Compute a standard basis $S=\{N_1,\ldots, N_r\}$ of $M+ \langle x_1, \cdots, x_s \rangle N$ w.r.t. $>$.

4. {\bf {Return:}} $S$.

\vskip 7pt

The following algorithm computes a $K$-basis of $M_{m,n}/\widetilde T_A(GA)$ and its codimension $ c = \dim(M_{m,n}/\widetilde T_A(GA))$.
\vskip 7pt
% % % % % % % % % %
{\bf{Algorithm 2:}} BasisCodimTangIm$_G$

{\bf{Input:}} $A=[a_{ij}]\in \mathfrak{m}M_{m,n}$, specification of $G$

{\bf{Output:}} matrices $M_1,\ldots, M_c\in M_{m,n}$, being a $K$-basis of $M_{m,n}/\widetilde T_A(GA)$ together with $c$, or -1 if $\dim_K(M_{m,n}/\widetilde T_A(GA))=\infty$.

1. Compute a standard basis $S=\{N_1,\ldots, N_r\}$ of $\widetilde T_A(GA)$ with Algorithm 1 for $G$.

2. Let $L_1,\ldots, L_r$ be the leading monomials of $N_1,\ldots, N_r$, respectively.

3. If $\dim_K\left(M_{m,n}/\langle  L_1,\ldots, L_r\rangle\right)=\infty$, 

\ \ \ \ {\bf {Return:}} -1

4. Else compute matrices $M_1,\ldots, M_c\subset M_{m,n}$ being a $K$-basis of $M_{m,n}/\langle L_1,\ldots, L_r\rangle$. 

5. {\bf {Return:}} $M_1\ldots, M_c, \ c$.

\begin{Remark}\rm
	The computation of a $K$-basis of $M_{m,n}/\langle L_1,\ldots, L_r\rangle$ is a combinatorial task. This basis is also a $K$--basis of $M_{m,n}/\widetilde T_A(GA)$. The {\sc Singular} (\cite {DGPS16}) command $K$-base computes the $K$-basis internally along these lines and returns a $K$-basis consisting of matrices with entries being monomials.
\end{Remark}
%\vskip 7pt

The following algorithm computes the {\em pre-determinacy bound}, i.e. the minimal $p$ such that $\mathfrak{m}^{p+1} M_{m,n}\subset \tilde T_A(GA)$, using a normal form algorithm NF w.r.t. a local monomial ordering (cf. \cite{GP07}). Then we compute a $G$--determinacy bound for $A$.
\vskip 7pt 

% % % % % % % % % %
{\bf{Algorithm 3:}} predeterm$_G$ 

{\bf Input:}  $A=[a_{ij}]\in \mathfrak{m}M_{m,n}$, specification of $G$

{\bf Output:}  integer $p$, the pre-determinacy bound for $A$ w.r.t. the group $G$, or -1 if codimension of $\widetilde T_A(GA)$ is infinite.

1. Compute a standard basis $S$ of the tangent image $\widetilde T_A(GA)$ by Algorithm 1 for $G$.

2. Compute the codimension $c$ of $\widetilde T_A(GA)$ by Algorithm 2.

3. If $c = -1$

\ \ \ \ {\bf {Return:}} --1.

4. Else loop

\hskip 30pt p:=0.

\hskip 30pt while (size($NF(\mathfrak{m}^{p+1}M_{m,n},S))>0$)

\hskip 50pt  p:=p+1.

5. {\bf {Return:}} p.

\vskip 7pt

% % % % % % % % % %
{\bf Algorithm 4:} determ$_G$

{\bf Input:} $A=[a_{ij}]\in \mathfrak{m}M_{m,n}$, specification of $G$

{\bf Output:} an integer $d$, a determinacy bound of $A$ w.r.t. the group $G$, or -1 if the codimension of $\widetilde T_A(GA)$ is infinite.

1. Compute $o=ord(A)$, the order of the matrix $A$.

2. Compute $p$, the pre-determinacy bound of $A$ w.r.t. $G$ by Algorithm 3.

3. If $p = -1$

\ \ \ \ {\bf {Return:}} --1.

4. Else compute $d=2p-o+2$.

5. {\bf {Return:}} $d$.

%%%%%%%%%%%%%%%%%%%%%%%
\section{Algorithms for the tangent space}\label{section 4}
%Let $G$ be one of the groups $\mathcal{R}$, $\mathcal{G}_l$, $\mathcal{G}_r$, and $\mathcal{G}_{lr}$, acting on $  M_{m,n}$. We are interested in computing the orbit $GA\subset   M_{m,n}$, for $A\in   M_{m,n}$. $G$ and $  M_{m,n}$ are not algebraic varieties, since they are of infinite dimension over $K$. Therefore we pass to the $k$-jet $G^{(k)}$ of $G$ acting on the $k$-jet $ M_{m,n}^{(k)}$ of $  M_{m,n}$ (see section \ref{section 2}). We identify $M_{m,n}^{(k)}$ with $Mat(m,n, K[{\bf x}]^{(k)})$, where $K[{\bf x}]^{(k)}$ denotes the vector space of polynomials of degree at most $k$. In the same way we identify $G^{(k)}$ with the subspace of $G$ consisting of polynomials of degree at most $k$. The action of $G^{(k)}$ on $M^{(k)}_{m,n}$ is a regular action of the affine algebraic groups $G^{(k)}$ on the affine space $ M_{m,n}^{(k)}$. 
%The orbit $G^{(k)}A$ is a locally closed subvariety of $ M_{m,n}^{(k)}$, and 

In this section we give an algorithm to compute equations for the closure of the orbit $G^{(k)}A\subset M_{m,n}^{(k)}$ for $G$ one of the groups $\mathcal{R}$, $\mathcal{G}_l$, $\mathcal{G}_r$, $\mathcal{G}_{lr}$, and for the codimension of $G^{(k)}A$ in $M_{m,n}^{(k)}$. The orbit $G^{(k)}A$ is a locally closed subvariety of the affine space $M_{m,n}^{(k)}$. Here an algebraic variety is, as usual,  considered as  a set with the Zariski topology over the algebraic closure $\bar K$, defined over $K$. 

For the application in section 5, we are only interested in the dimension of the tangent space to  $G^{(k)}A$ at $A$. However, in positive characteristic the tangent space may not coincide with the tangent image and therefore we cannot use the algorithms of the previous section.
We do not know any other method to compute the dimension of the tangent space to the orbit, except by computing the dimension of the orbit itself (in positive characteristic).
\vskip 7pt

The following theorem is the basis for our applications.
\begin{Theorem} \label{th4.1}
	Let $G$ be any of the groups  $\mathcal{R}$, $\mathcal{G}_l$, $\mathcal{G}_r$, and $\mathcal{G}_{lr}$, acting on $  M_{m,n}$. Assume that the tangent image $\widetilde T_A(GA)$ has finite codimension in $M_{m,n}$. Let $p$ be the pre-determinacy bound for $A$.  For $k \geq p$ the following holds:
	\begin{enumerate}
		\item $A$ is $G \  (2k -ord(A) + 2)$--determined.
		\item $\dim_K M_{m,n}/\widetilde T_A(GA) = \dim_K M^{(k)}_{m,n}/\widetilde T^{(k)}_A(GA)$.
		\item The tangent space $T_A(GA)$ to the orbit $GA$ has finite codimension in $  M_{m,n}$.
		\item $\dim_K   M_{m,n}/T_A(GA)=\dim_K M_{m,n}^{(k)}/T^{(k)}_A(GA) =  \dim M_{m,n}^{(k)} - \dim G^{(k)}A$.
		
		\item $\dim G^{(k)}A=\dim G^{(k)}-\dim G_A^{(k)}$, where $G_A^{(k)}$ denotes the stabilizer of $A$ in $G^{(k)}$.
	\end{enumerate}
\end{Theorem}

\begin{proof}
1. Was proved in \cite{GP16}. \\
2. This follows from $\mathfrak{m}^{k+1}  M_{m,n}\subset \widetilde T_A(GA)$. \\
3. Follows by assumption from  $\widetilde T_A(GA)\subset T_A(GA) \subset M_{m,n}$. \\
4. The first equality follows from $\mathfrak{m}^{k+1}  M_{m,n}\subset T_A(GA)$,
 the second since the orbit is smooth and has the same dimension as its tangent space.	\\
5. Is well known (cf. e.g. \cite[Theorem 3.7]{FSR05}). The dimension of an affine variety means the Krull dimension of its coordinate ring.
\end{proof}

Since the tangent image of an algebraic group action coincides with the tangent space iff the orbit map is separable (cf. \cite[Theorem 3.7]{FSR05}), we get:

\begin{Corollary}
 With the assumptions of Theorem \ref {th4.1} the following are equivalent:
 	\begin{enumerate}
		\item The orbit map $o: G^{(k)}\to G^{(k)}A, \ g\mapsto jet_k(gA)$, is separable.
                \item  $\widetilde T_A(GA)=  T_A(GA)$.
                \item $\dim_K M_{m,n}/\widetilde T_A(GA) = \dim M^{(k)}_{m,n} - \dim G^{(k)} + \dim G^{(k)}_A$
        \end{enumerate}
\end{Corollary}
\medskip

$ M^{(k)}_{m,n}$ is an affine space of dimension $t=mn\binom{s+k}{k}$ with coordinate ring $K[u]=K[u_1,\ldots, u_t]$ and $G^{(k)}A$ is a (locally closed) subvariety of $ M^{(k)}_{m,n}$ of a certain dimension which we want to know. 
	
	Our first algorithm computes polynomials $F_1,\ldots, F_r\in K[u]$ by elimination defining $G^{(k)}A$ set theoretically. Then we compute the dimension by computing a standard basis of $\langle F_1,\ldots, F_r\rangle$. This approach has the disadvantage that the dimensions of $G^{(k)}$ and $M^{(k)}_{m,n}$ are quite big, already for small $m, n, k$, and we have to compute standard basis with respect to an elimination ordering in rings with many variables.
	
	Our second algorithm computes polynomials defining the stabilizer $G^{(k)}_A\subset G^{(k)}$ of $A$ and its dimension can be computed by a standard basis w.r.t. any ordering. 
	This algorithm involves less variables and is preferred if we are only interested in the dimension of orbit and not its equations. 
\vskip 7pt
	
	The main challenge is to put the equations in the right form. Let us first consider the right group $G^{(k)}=\mathcal{R}^{(k)}$. An element $g\in \mathcal{R}^{(k)}$ is given by $s$ polynomials in $K[{\bf x}]^{(k)} = jet_k(K[x_1,\ldots, x_s])$, which can be written as
	\begin{align*}
	g_i(x_1,\ldots, x_s)=\sum\limits_{j=1}^s(\delta_{ij}+g_{ij})x_j+\sum\limits_{|a|=2}^kh_{ia}{\bf x}^a, \hskip 3pt a=(a_1,\ldots, a_s),\tag{*}
	\end{align*}
	with $\delta_{ij}$ the Kronecker symbol, $1=[\delta_{ij}]$ the identity matrix and $\det(1+g_{ij})\ne 0$.
	
	If $A=[a_{ij}({\bf x})]\in M_{m,n}^{(k)}$ then $g$ acts on $A$ by substitution and taking $k$-jets, i.e.
	$$gA=jet_k[a_{ij}(g)]=[jet_k(a_{ij}(g_1({\bf x}),\ldots, g_s({\bf x}) ))].$$
	Let $G_{ij}, H_{ia}$, $i,j=1,\ldots, s$, $2\le|a|\le k$, be new variables. Then the group $\mathcal{R}^{(k)}$ is the affine variety defined as the complement of $\det(1+G_{ij})=0$ in the affine space of dimension $\dim \mathcal{R}^{(k)} = s\binom{s+k}{k}-s$ with coordinates $G=(G_{ij}, H_{ia})$.
\vskip 7pt
	
	If $G^{(k)}=\mathcal{G}^{(k)}_{lr}$ then an element $g\in G^{(k)}$ is given by $g_i(x_1,\ldots, x_s)$ as in (*) above and in addition by matrices
	$$[u_{ij}({\bf x})]\in GL(m, K[{\bf x}]^{(k)}), \hskip 10pt  i,j=1,\ldots,m,$$
	$$u_{ij}=\delta_{ij}+u_{ij0}+\sum\limits_{|a|=1}^k u_{ija}{\bf x}^a, \hskip 10pt \det(1+[u_{ij0}])\ne 0.$$
	$$[v_{ij}({\bf x})]\in GL(n, K[{\bf x}]^{(k)}), \hskip 10pt  i,j=1,\ldots,n,$$
	$$v_{ij}=\delta_{ij}+v_{ij0}+\sum\limits_{|a|=1}^k v_{ija}{\bf x}^a, \hskip 10pt \det(1+[v_{ij0}])\ne 0.$$
	$([u_{ij}]$, g, $[v_{ij}])$ acts on $ M^{(k)}_{m,n}$ by
	$$([u_{ij}], g, [v_{ij}])A=jet_k([u_{ij}] [a_{ij}(g)] [v_{ij}]).$$
	$\mathcal{G}_{lr}^{(k)}$ is an  affine variety of dimension $N=(m^2+n^2+s)\binom{s+k}{k}-s$. If $U_{ija}$, $i,j\in \{1,\ldots,m\}$, $0\le |a|\le k$ and $V_{ija}$, $i,j\in \{1,\ldots,n\}$, $0\le |a|\le k$ are new variables, then $\mathcal{G}_{lr}^{(k)}$ is the complement of the hypersurface
	$$\det:=\det(1+[G_{ij}])\det(1+[U_{ij0}])\det(1+[V_{ij0}])=0$$
	in the affine space $K^N$ with coordinates $U=(U_{ija})$, $i,j\in \{1,\ldots,m\}$, $0\le |a|\le k$, $V=(V_{ija})$, $i,j\in \{1,\ldots,n\}$, $0\le |a|\le k$,  and $G=(G_{ij}, H_{ia})$ $i,j\in \{1,\ldots,n\}$, $2\le |a|\le k$.
	
	The other groups $G^{(k)}$ are special cases of $\mathcal{G}^{(k)}_{lr}$.
\vskip 7pt
	
	By our choice of coordinates $U_{ij0}$, $V_{ij0}$, and $G_{ii}$, the groups $G^{(k)}$ pass through $0\in K^N$ where $0$ corresponds to the identity in $G^{(k)}$. This allows us to compute in the polynomial ring $K[U,V,G]$ as well as in the localization $K[U,V,G]_>$ with $>$ a local ordering. Note that det  is a unit in $K[U,V,G]_>$, hence the condition det $\neq 0$ is automatic in this ring.
\vskip 7pt
	
	The algorithm for describing the orbit of $G^{(k)}=\mathcal{G}_{lr}^{(k)}$ can be described as follows:
	
	\vskip 7pt
%%%%%%%%%%%%%%%%
	{\bf Algorithm 5:} OrbitEq ($ G=\mathcal{G}_{lr}$)
	
	{\bf Input:}  integer $k$, matrix $A=[a_{ij}]\in M_{m,n}^{(k)}$
	 
		{\bf Output:} polynomials $F_1,\ldots, F_r\in K[u_1,\ldots, u_t]$, $t= \dim M^{(k)}_{m,n}$, such that the variety $V(F_1,\ldots, F_s)$ coincides with the closure of $\mathcal{G}_{lr}^{(k)}A$.
		
1. In the polynomial ring $K[{\bf x}, {U}, V, G]=K[x_1,\ldots, x_s, U_{ija}, V_{ija}, G_{ij}, H_{i,a}]$ define the polynomials
$$G_i=\sum\limits_{j=1}^s(\delta_{ij}+G_{ij})x_j+\sum\limits_{|a|=2}^kH_{ia}{\bf x}^a, \hskip 10pt i=1,\ldots, s,$$
$$U_{ij}=\delta_{ij}+U_{ij0}+\sum\limits_{|a|=1}^kU_{ija}{\bf x}^a\hskip 10pt i,j=1,\ldots, m,$$
$$V_{ij}=\delta_{ij}+V_{ij0}+\sum\limits_{|a|=1}^kV_{ija}{\bf x}^a\hskip 10pt i,j=1,\ldots, n.$$

2. Construct the matrix $B=[b_{ij}]\in Mat(m,n, K[{\bf x}, U, V, G])$ with
$$b_{ij}=jet_k\left([U_{ij}] \cdot [a_{ij}(G_1,\ldots, G_s)] \cdot [V_{ij}]\right),$$
where the $k$-jet is taken w.r.t. ${\bf x}$.

3. Write the polynomials $b_{ij}$ as $\sum\limits_{|a|=0}^k c_{ija}{\bf x}^a$
with coefficients $c_{ija}\in K[U, V, G]$. 

Let $C$ denote the ideal in $K[U,V,G]$ generated by the $t$ polynomials $c_{ija}$, $i=1,\ldots,m$, $j=1,\ldots, n$, $0\le |a|\le k$. Denote the polynomials $c_{ija}$ by $c_1,\ldots, c_t$.

4. Let $D\subset K[U,V,G,u]$ be the ideal generated by the polynomials
$$u_i-c_i(U,V,G), \hskip 10pt i=1,\ldots, t.$$

Eliminate the variables $U_{ij}$, $V_{ij}$, $G_{ij}$, $H_{i,a}$ from $D$ by computing a standard basis w.r.t. an elimination ordering (see \cite{GP07}). Get finitely many polynomials $F_1,\ldots, F_r\in K[u]$.

5. {\bf Return:} $F_1,\ldots, F_r$.
%\vskip 7pt
\begin{Remark}\rm
	The algorithm terminates since each of the five steps terminates obviously. It is correct since the $\mathcal{G}^{(k)}_{lr}$ is the open subset $\det\ne 0$ in $K^N$ with coordinates $U, V, G$ and the orbit map $o: G^{(k)}\to  M_{m,n}^{(k)}$ is given on the ring level by 
	$$u_i=c_i(U,V, G), \hskip 7pt i=1,\ldots, t.$$
	It is well known (e.g. \cite{GP07}) that the closure of image is defined by eliminating $U, V, G$ from the ideal $\langle u_i-c_i \rangle$.
\end{Remark}

	The computation of the equation for the orbit of the other groups $\mathcal{R}$, $\mathcal{G}_l$, $\mathcal{G}_r$ is a special case of the computation for  $\mathcal{G}_{lr}$, by omitting $U_{ij}$ and $V_{ij}$ for $\mathcal{R}$, $V_{ij}$ for $\mathcal{G}_l$ and $U_{ij}$ for $\mathcal{G}_r$.

\vskip 7pt
	We present now an algorithm to compute the stabilizer 
	$$G_A^{(k)}=\{g\in G^{(k)}\mid gA=A\}$$
	of the action of $G^{(k)}$ on $M_{m,n}^{(k)}$. We use the notations from Algorithm 5.
	
\vskip 7pt
%%%%%%%%%%%%%%%%
	
	{\bf Algorithm 6:} StabEq ($G=\mathcal{G}_{lr}$.)
	
	{\bf Input:} integer $k$, matrix $A=[a_{ij}]\in M_{m,n}^{(k)}$
	
	{\bf Output:}  polynomials $D_1,\ldots, D_t\in K[U,V,G]$ such that the variety $V(D_1,\ldots, D_t)$ coincides with $G_A^{(k)}\subset G^{(k)}$.
		
1. and 2. are the same as in algorithm Orbit equations, we get  
 $$B=[b_{ij}]\in Mat(m,n, K[{\bf x}, U, V, G]).$$	
 
3. Write 	
	$$b_{ij}-a_{ij}=\sum\limits_{|a|=0}^kd_{ija}{\bf x}^a, \hskip 10pt d_{ija}\in K[U,V,G]$$
	and let $D$ be the ideal in $K[U,V,G]$ generated by $d_{ija}$, $i=1,\ldots, m$, $j=1,\ldots,n$, $|a|=0,\ldots, k$. Denote the polynomials $d_{ija}$ by $D_1,\ldots, D_t$.
	
	4. Return $D_1\ldots, D_t$.
%	\vskip 7pt
\begin{Remark}\rm
	The algorithm terminates and since the coefficients $d_{ija}$ of $b_{ij}-a_{ij}$ are the polynomials defining the set $\{gA-A\mid g\in G^{(k)}\}$ it is also correct.
\end{Remark}\rm	
%	\vskip 7pt
	It is now easy to compute the codimension of $T_A(GA)$ in $M_{m,n}$ if $\tilde T_A(GA)$ has finite codimension. We have two algorithms.
\vskip 10pt

%%%%%%%%%%%%%%%%
	{\bf Algorithm 7:} codimTang$_G$1
	
	{\bf Input:} $A=[a_{ij}]\in M_{m,n}$, assume $\dim_K(M_{m,n}/\tilde T_A(GA))<\infty$, specification of $G$.
	
	{\bf Output:} $c=\dim_K M_{m,n}/T_A(GA)$
	
	(1) Compute the pre-determinacy bound $p$ for $A$ with Algorithm 3.
	
	(2) Apply Algorithm 5 to compute the orbit equations $F_1,\ldots, F_r$.
	
	(3) Compute a standard basis of the ideal $I=\langle F_1,\ldots, F_r\rangle$ w.r.t. any monomial ordering.
	
	(4) Compute $\dim I$.
	
	(5) {\bf Return:} $c=\dim M_{m,n}^{(p)} -\dim I$ (note: $\dim M_{m,n}^{(k)}=mn\binom{s+k}{k})$
	
	\vskip 10pt
%%%%%%%%%%%%%%%%
	
	{\bf Algorithm 8:} codimTang$_G$2
	
	{\bf Input:} $A=[a_{ij}]\in M_{m,n}$, assume $\dim_K M_{m,n}/\tilde T_A(GA)<\infty$, specification of $G$.
	
	{\bf Output:} $c=\dim_K M_{m,n}/T_A(GA).$
	
	(1) Compute the pre-determinacy bound $p$ for $A$ with Algorithm 3.
		
	(2) Compute stabilizer equations $D_1,\ldots, D_t$ with Algorithm 6.
	
	(3) Compute a standard basis of the ideal $D=\langle D_1,\ldots, D_t\rangle$ w.r.t. any monomial ordering.
	
	(4) Compute $\dim D$.
	
	(5) {\bf Return:}  $c=\dim M_{m,n}^{(p)} - \dim G^{(p)}+ \dim D$
	%$c=\dim G^{(p)}-\dim D$ 
	(e.g. $\dim \mathcal{G}_{lr}^{(k)}=(m^2+n^2+s)\binom{s+k}{k}-s$).
	
	% % % % % % % % % % % % % % % % % % % % %	
	\section{A problem in singularity theory}
	The algorithms of this paper are of interest for the classification of singularities in arbitrary characteristic.
	The classification of isolated hypersurface singularities $f\in\C\{x_1,\ldots, x_n\}$ has a long tradition with contribution by many authors, most notably by V.I. Arnold and his school \cite{AGV85}. For formal power series $f\in R=K[[{\bf x}]]$, where $K$ is a field of arbitrary characteristic, the classification started with \cite{GK90} and was continued in \cite{BGM11}, \cite{BGM12}, \cite{GN14}, \cite{Ng15}. The two most important equivalence relations are {\em right equivalence} and {\em contact equivalence}. Here $f,g\in R$ are right equivalent ($f\mathop\sim\limits^r g$) if $f=\phi(g)$ for some $\phi\in Aut(R)$ and they are contact equivalent ($f\mathop\sim\limits^c g$) if $f=u\cdot\phi(g)$ for some $\phi\in Aut(R)$ and a unit $u\in R^*$. An indispensable assumption for the classification is that the power series are finitely determined (for the considered equivalence relation) and that an explicit and computable determinacy bound is known.
	
	If the field $K$ has characteristic $0$ (or in the case of convergent power series over $\C$ and $\R$) such bounds are known for a long time (see e.g. \cite{GLS07} for references)
and finite determinacy is equivalent to $f$ having an isolated singularity. Moreover, in this case the orbit map is separable and the action of the right group $\mathcal{R}=Aut(R)$ and the contact group $\mathcal{K}=R^*\ltimes \mathcal{R}$ on $R$ can be faithfully described on the tangent level.
\vskip 7pt

In positive characteristic however one has to work directly with the group actions and therefore the methods of proof must be different. This is in principal well-known, but it has been discovered only recently by the authors that the tangent space to the orbit of the action of $\mathcal{K}$ on $K[[{\bf x}]]$ may be different from the tangent image, a fact that had been overlooked by several authors before. The purpose of this section is to prove this fact by giving the details of the computation of an explicit example, as announced in \cite{GP16}.

Note that the tangent space to the orbit coincides with the tangent image iff the orbit map $\mathcal{K}^{(k)}\to \mathcal{K}^{(k)} f$, $(u,\phi)\mapsto u\phi(f)$, is separable for sufficiently big $k$ (see \cite{GP16}) for a discussion and a precise statement), which is always true in characteristic 0. 
It came as a surprise to us that this separability may fail, since it was shown in \cite{BGM12} that the map of the full action, $\mathcal{K}^{(k)}\times K[[{\bf x}]]^{(k)}\to K[[{\bf x}]]^{(k)}$, is always separable. Our experiments with {\sc Singular} (\cite {DGPS16}), using the algorithms of this paper, show that also in positive characteristic separability of the orbit map holds in most cases and in fact, for the right group $\mathcal{R}$ we do not have a  non-separable example with isolated singularity so far. 
\vskip 7pt

Right resp. contact equivalence for power series is a special case for matrices of size $m=n=1$ and the groups $\mathcal{R}$ resp. $\mathcal{K} = \mathcal{G}_l$. Our algorithms go however much further by treating matrices of arbitrary size and more equivalence relations. They provide general tools, not only to compute determinacy bounds, but also to decide in concrete cases in positive characteristic whether the tangent image coincides with the tangent space, i.e. whether separability of the orbit map holds or not. 

The classificaion of general matrices with a small number of moduli is still an unsolved problem and we believe that the presented algorithms are useful in this context.

\vskip 7pt

\begin{Example}\label{counterexmple}\rm
We give an example for $G=\mathcal{K}$ acting on $ K[[x,y]]$, where the tangent image is strictly contained in the tangent space.

Let $\characteristic(K)=2$, $f = x^2 + y^3 \in K[[x,y]]$. We compute:

\begin{itemize}
	\item the tangent image $\widetilde T_{f}(\mathcal{K} f) = \langle f \rangle + \mathfrak{m} \cdot {\rm{j}}(f)$ is $\langle x^2,xy^2,y^3 \rangle$,\\
		its codimension in $M_{1,1}=K[[x,y]]$ is $c=5$.

	\item $f$ is $\mathcal{K}$ 4-determined, its $\mathcal{K}$--pre-determinacy is $p=2$.

	\item  $K[[x,y]]^{(p)} = K[[x,y]]/ \mathfrak{m}^3$ has dimension $t=6$. 

	\item The group $\mathcal{K}^{(p)}$ has dimension 16 and the stabilizer of $jet_p(f)$ has dimension 14.
	
	\item	The dimension of the orbit is 2, its codimension (also the codimension of the tangent space) is 4. As the codimension of the tangent image is 5, the orbit map $\mathcal{K}^{(p)} \to \mathcal{K}^{(p)}f$
	is not separable.
	.
\end{itemize}
\end{Example}

In order to prove the statements, we present the {\sc Singular} input for direct use, together with some comments.

\begin {verbatim} 
ring r = 2,(x,y),ds;                   //local ordering, char(K)=2 
poly f = x2 + y3;

//compute tangent image (Algo 1) and its codimension (Algo 2)
ideal j = jacob(f);
ideal m = maxideal(1);
ideal T = std(m*j+ideal(f)); T;        //T=tangent image =<x2,xy2,y3>
int c = vdim(T); c;                    //c=codim of tangent image =5

//Compute pre-determinacy (Algo 3) and determinacy bound (Algo 4) 
int p, d; 
while (size(NF(maxideal(p+1),T))!=0)  
{ p = p+2; }
p;                                     //p=pre-determinacy =2 
d = 2*p-ord(f)+2; d;                   //d=determinacy bound =4

//Compute orbit equations (Algo 5, step 1. and 2.)
ideal km = kbase(maxideal(p+1));       
int t = size(km); t;                    //t=dim of p-jet of M_m,n =6
int s = t-1;                            //we will omit km[6] = 1
ring R = (2,u(0..s),a(1..s),b(1..s)),(x,y),ds; 
                                        //ring for creating contact group
poly f = imap(r,f);
ideal km = imap(r,km);
poly u,h1,h2; 
int ii;                                                      
for (ii=1; ii<=s; ii++)
{ u  = u+km[ii]*u(ii);
  h1 = h1+km[ii]*a(ii);
  h2 = h2+km[ii]*b(ii);
}
u = u + u(0) + 1;                       //u=1+ u(0)+u(1)x+..., unit
h1 = h1 + x;
h2 = h2 + y;                            //(h1,h2)= id+..., coord. change
map phi = (r,h1,h2); phi;
poly F = jet(phi(f),p);
F = jet(u*F,p); F;                      //orbit equations

//Compute stabilizer (Algo 6)   
F = jet(F-f,p);                         
matrix C = coef(F,xy);
int n = ncols(C);       
ideal D = C[2,1..n]; //coefficients of F-f

ring S = 2,(u(0..s),a(1..s),b(1..s)),ds; //local ring of contact group
ideal D = imap(R,D);  D;                 //ideal of stabilizer

//Compute codimension of tangent space (Algo 8)
D = std(D);     
int c1 = dim (D); c1;                  //c1=dimension of stablizer =14
int c2 = nvars(S); c2;                 //c2=dimension of group =16
int c3 = c2-c1;  c3;                   //c2=dimension of orbit =2 
int c4 = t-c3; c4;                     //c4=codimension of orbit =4
c-c4;                                  //= 1, orbit map not separable

\end{verbatim}

%%%%%%%%%%%%%%%%%%%%%%%

%\nocite{*}
%\bibliographystyle{amsalpha}
%\bibliography{Phamthuyhuong2}

\providecommand{\bysame}{\leavevmode\hbox to3em{\hrulefill}\thinspace}
\providecommand{\MR}{\relax\ifhmode\unskip\space\fi MR }
% \MRhref is called by the amsart/book/proc definition of \MR.
\providecommand{\MRhref}[2]{%
	\href{http://www.ams.org/mathscinet-getitem?mr=#1}{#2}
}
\providecommand{\href}[2]{#2}

Fachbereich Mathematik, Universit\"at Kaiserslautern, Erwin-Schr\"odinger Str.,
67663 Kaiserslautern, Germany

E-mail address: greuel@mathematik.uni-kl.de

\vskip 7pt 

Department of Mathematics, Quy Nhon University, 170 An Duong Vuong
Street, Quy Nhon, Vietnam

Email address: phamthuyhuong@qnu.edu.vn

\end{document}